\documentclass{amsart}

\usepackage{hyperref}
\usepackage{amsthm}
\usepackage{amssymb}
\usepackage{color}
\usepackage{enumerate}
\usepackage{mathrsfs}
\usepackage{thmtools}
\usepackage{lineno}
\usepackage{tikz-cd}

\theoremstyle{plain}
\newtheorem{theorem}{Theorem}[section]
\newtheorem{lemma}[theorem]{Lemma}
\newtheorem{proposition}[theorem]{Proposition}
\newtheorem{corollary}[theorem]{Corollary}

\theoremstyle{definition}
\newtheorem{definition}[theorem]{Definition}
\newtheorem{example}[theorem]{Example}
\newtheorem{conjecture}[theorem]{Conjecture}

\theoremstyle{remark}
\newtheorem{remark}[theorem]{Remark}

\newcommand{\C}{\mathbf{C}}

\renewcommand{\O}{\mathcal{O}}
\renewcommand{\P}{\mathbf{P}}
\newcommand{\Q}{\mathbf{Q}}
\newcommand{\Z}{\mathbf{Z}}

\newcommand{\sco}{\mathcal{O}}

\renewcommand{\t}{\sigma}

\newcommand{\gal}{\mathrm{Gal}}

\newcommand{\per}{\mathrm{Per}}

\newcommand{\gdr}{Galois--dynamics correspondence}

\newcommand{\set}[1]{\left\{#1\right\}}

\newcommand\restr[2]{{% we make the whole thing an ordinary symbol
  \left.\kern-\nulldelimiterspace % automatically resize the bar with \right
  #1 % the function
  \right|_{#2} % this is the delimiter
  }}

\begin{document}

\title{A Galois--dynamics correspondence for unicritical polynomials}

\author{Robin Zhang}
\address[Robin Zhang]{Department of Mathematics, Columbia University}
\email{rzhang@math.columbia.edu}

\date{May 18, 2021}

\begin{abstract}
	In an analogy with the Galois homothety property
	for torsion points of abelian varieties that was
	used in the proof of the Mordell--Lang conjecture,
	we describe a correspondence between the action of a 
	Galois group and the dynamical action of a rational map.
	For nonlinear polynomials with rational coefficients,
	the irreducibility of the associated
	dynatomic polynomial serves as a convenient criterion, although we
	also verify that the correspondence occurs in several cases
	when the dynatomic polynomial is reducible.
	The work of Morton, Morton--Patel, and Vivaldi--Hatjispyros
	in the early 1990s
	connected the irreducibility and Galois-theoretic properties
	of dynatomic polynomials to rational periodic
	points;
	from the Galois--dynamics
	correspondence, we derive similar consequences
	for quadratic periodic points of unicritical polynomials.
	This is sufficient to deduce the
	non-existence of quadratic periodic points of quadratic polynomials
	with exact period
	$5$ and $6$, outside of a specified finite set from Morton
	and Krumm's work in explicit Hilbert irreducibility.
\end{abstract}

\maketitle

\section{Introduction}
\label{sec:intro}

\subsection{A general definition}
\label{subsec:def}
In arithmetic geometry, it is natural to consider
the action of Galois groups on torsion points of
abelian varieties.
The proof of the Mordell--Lang conjecture crucially
uses the fact that torsion points of abelian varieties always satisfy
the following Galois homothety property (cf.
\cite{lang-division},
\cite[Th\'{e}or\`{e}me~2]{serre-cours},
\cite[Lemme~12]{hindry},
\cite[Theorem~2.1.6]{mcquillan},
\cite[Section~1]{poonen-mordell-lang}):
for an abelian variety $A$ defined over a number field $k$,
there exists an $i \geq 1$ such that for all positive integers $m$,
there is a $\sigma_m \in \gal(\overline{k}/k)$ such that
\[
	\sigma_m(x) = m^i x
\]
for all $\overline{k}$-rational points $x$ on $A$ of finite order
coprime to $m$.

For arithmetic dynamical systems, we introduce an analogous
correspondence with two goals in mind. First, we establish a
``dynamical Galois homothety property'', designated the
\textit{\gdr}, as a weaker and slightly more prevalent
condition than the irreducibility of
dynatomic polynomials studied in the work of
Vivaldi--Hatjispyros~\cite{vh}. 
Second, we explore the connections between the {\gdr}
and quadratic periodic points.

Let $\phi:\P^1 \rightarrow \P^1$ be a rational map over a
field $k$ and let $\phi^N$
denote the $N$-th iterate of $\phi$, so $\phi^N := \phi
\circ \phi^{N - 1}$. For $S$ a subset of an algebraic closure
$\overline{k}$ of $k$, let $\per_{N,S}(\phi)$ denote the periodic
points in $S$ of $\phi$ of exact period $N$.

\begin{definition}
\label{def:A}
	Let $\phi:\P^1 \rightarrow \P^1$ be a rational map over a
	field $k$, $N$ be an integer greater than $1$, and
	$K/k$ be a nontrivial Galois extension.
	For a periodic point $z \in \per_{N,K-k}(\phi)$,
	the tuple $(\phi, N, K/k, z)$ satisfies the
	\textit{{\gdr}} (GDC) if and only if
	there is a positive integer $i < N$ and a nontrivial
	$\sigma \in	\gal(K/k)$ such that
	\[
		\sigma (z) = \phi^i(z).
	\]
	\noindent
	Furthermore, we say that the triple $(\phi, N, K/k)$ satisfies the
	{\gdr} if and only if $(\phi, N, K/k, z)$ satisfies the {\gdr} for
	every $z \in \per_{N,K-k}(\phi)$.
\end{definition}

Unlike the Galois homothety property, which always holds in the
setting of abelian varieties, the {\gdr} is not always satisfied
for dynamical systems.
Even if we restrict our attention to unicritical polynomials,
which can be written as $\phi_{d, c}(z) := z^d + c$,
there are instances in which the {\gdr} might
not be satisfied.

\begin{example}
\label{ex:fail}
	If $K$ is the splitting field of $\phi_{2, t}^3(z) - z$ over the
	rational function field $k=\C(t)$, then the directed graph
	of periodic points of $\phi_{2, t}$ of exact period
	$N = 3$ consists of two disjoint $3$-cycles. A result of
	Bousch~\cite[Chapter 3, Theorem 3]{bousch} implies that the
	Galois group $\gal\big(K/\C(t)\big)$ properly contains the full
	automorphism group of this graph and, in particular, contains an
	element $\sigma$ of order $2$ that interchanges the two $3$-cycles.
	If $k$ is the fixed field $K^\sigma$, then $[K:k] = 2$. For
	any $z \in \per_{N,K-k}(\phi)$ of exact period $3$, the
	$\gal(K/k)$-orbit of $z$ is $\{z, \sigma (z)\}$, but $\sigma (z)$
	is not in the forward orbit of $z$ with respect to $\phi_{2, t}$.
	Hence $(\phi_{2, t}, 3, K/K^\sigma)$ does
	not satisfy the {\gdr}.
\end{example}

\begin{remark}
	The definition of the {\gdr} can be naturally generalized
	by allowing $\phi$
	to be any endomorphism on an algebraic variety.
	If $A$ is an abelian variety and $[m]$ denotes its
	multiplication-by-$m$ endomorphism, then
	the Galois homothety property due to Faltings and Serre
	(cf.~\cite[Theorem~2.1.6]{mcquillan}) implies that
	for all integers $m$ and $N$ greater than $1$, 
	$([m], N, \overline{k}/k)$ satisfies the {\gdr}.
\end{remark}

\begin{remark}
	\label{rem:disjoint}
	If $\phi$ is a polynomial with coefficients in $k$, then its iterative
	action commutes with the action of $\gal(K/k)$.
	In particular, $(\phi, N, K/k, z)$ satisfies the {\gdr}
	if and only if the $N$-cycle of $z$ is not disjoint from
	its conjugates by some element of the Galois group, i.e.
	\[
		\set{z, \ldots, \phi^{N-1}(z)} \cap \set{\sigma (z), \ldots, \sigma (\phi^{N-1}(z))}
			\neq \emptyset,
	\]
	for some nontrivial $\sigma \in \gal(K/k)$.
\end{remark}

\subsection{Irreducibility and the main theorem}
\label{subsec:intro-irred}
We will focus on the classical case when $k = \Q$ and
$\phi$ is a nonlinear polynomial in $\Q[z]$. First,
we recall the definition of dynatomic polynomials.
An $N$-periodic point $z$ of $\phi$
is always a root of the polynomial $\phi^n(z) - z$ for all
multiples $n$ of $N$. By the M\"obius inversion formula, we have
a factorization of $\phi^n(z) - z$ in terms of the $N$-th
\emph{dynatomic polynomial} $\Phi_N(z)$,
\begin{align*}
	\phi^n(z) - z &= \prod_{N|n} \Phi_N(z) \in \Q[z] \\
	\Phi_N(z) :&= \prod_{m|N}(\phi^m(z) - z)^{\mu(N/m)} \in
		\Q[z],
\end{align*}
where $\mu$ is the M\"{o}bius function.
The roots of $\Phi_N(z)$ are called the periodic points of
$\phi$ of formal period $N$, which generically are
the periodic points of $\phi$ of exact period $N$.
When we study the
family of unicritical polynomials, with $\phi = \phi_{d, t}$, we can
consider the dynatomic polynomial $\Phi_N(z, t)$
as an element of $\Q[z, t]$.
The zero locus of $\Phi_N(z)$ defines an affine curve
with a smooth projective model denoted $C_1(N)$
that carries an action of $\Z/N\Z$
induced by the iteration of $\phi$.
We can then define the quotient $C_0(N)$ of
$C_1(N)$ by this action.

The {\gdr} holds in the case considered by
Vivaldi--Hatjispyros~\cite{vh}, namely when $\phi$ is any
nonlinear polynomial in $\Q[z]$ and its associated dynatomic polynomial
$\Phi_N(z)$ is irreducible in $\Q[z]$.
In Section~\ref{subsec:crit}, we slightly modify
a result of Vivaldi--Hatjispyros
to obtain the following
irreducibility criterion for the {\gdr}.

\begin{proposition}[{Proposition~\ref{prop:vh2}}]
\label{prop:vh}
	Let $\phi$ be any nonlinear polynomial and let $N$ be an integer
	greater than $1$. If the dynatomic polynomial $\Phi_N(z)$
	is irreducible in $\Q[z]$ then
	$(\phi, N, K/\Q)$ satisfies the {\gdr} for all
	nontrivial Galois extensions $K/\Q$.
\end{proposition}

Therefore, the {\gdr} can be viewed as a weaker notion than the
irreducibility of $\Phi_N(z)$. Irreducibility is sufficient
to provide most of the cases of the {\gdr}, since it is often the
``typical'' situation.
For $\phi$ belonging to certain families of polynomials in $\Q[z]$,
including the unicritical polynomials $\phi_{d, c}$, the dynatomic
polynomial $\Phi_N(z, t)$ is known to be irreducible in $\C(t)[z]$
by the work of Bousch~\cite[Chapitre 3, Th\'{e}or\`{e}me 1]{bousch},
Lau--Schleicher~\cite[Theorem~4.1]{lau-schleicher},
and Morton~\cite[Theorem~B]{morton-irred}.
So for these families of polynomials, we apply
Hilbert's irreducibility theorem in Section~\ref{subsec:exceptional-set}
to show that $(\phi_{d, c}, N, K/\Q)$ satisfies the {\gdr} for all
Galois extensions $K/\Q$ so long as $c$ lies outside of a subset of
$\Q$ of asymptotic density $0$,
called the \textit{exceptional set} $\Sigma_{d, N}$.

The main theorem of this paper establishes cases of the {\gdr} for
quadratic polynomials with various $N$, even when
$\Phi_N$ is reducible and $c$ is in the exceptional set $\Sigma_{2, N}$.
Its proof in Section~\ref{sec:main-proof} combines the irreducibility
criterion of Proposition~\ref{prop:vh} with descriptions of the
exceptional sets $\Sigma_{2, N}$ due to Morton~\cite{morton-1992} and
Krumm~\cite{krumm-galois, krumm-finiteness} to establish the {\gdr}
when $c \in \Q - \Sigma_{2, N}$. For the exceptional
cases $c \in \Sigma_{2, N}$, an explicit understanding of how
$\Phi_N(z, c)$ factors for fixed specializations
$t = c$ is used to check the {\gdr} directly on a case-by-case basis.

\begin{theorem}
	\label{thm:gdr-d2}
	For the quadratic polynomial $\phi_{2,c}(z) = z^2 + c$ with rational
	coefficients and all nontrivial Galois extensions $K/\Q$,
	$(\phi_{2, c}, N, K/\Q)$
	satisfies the {\gdr} in the following cases:
	\begin{itemize}
		\item $N = 2$: all $c \in \Q$;
		\item $N = 3$: all $c \in \Q$;
		\item $N = 4$:
			\begin{itemize}
				\item $[K:\Q] = 2$: all $c \in \Q$;
				\item $[K:\Q] > 2$: all $c \notin
					\set{\frac{-s^3 - 2s + 4}{4s} \middle| s \in \Q^\times}$;
			\end{itemize}
		\item $N = 5, 6, 7$, or $9$: all but finitely many $c \in \Q$.
	\end{itemize}
\end{theorem}

\subsection{Quadratic points and periodic points}

Theorem~\ref{thm:gdr-d2} has a nice consequence in the
dynamics of quadratic polynomials over quadratic number fields.
For periods $5$ and $6$, no periodic points of $\phi_{2, c}$
are believed to exist in quadratic number fields besides
the single known $6$-cycle, as suggested by extensive numerical
evidence by Doyle--Faber--Krumm~\cite{doyle-faber-krumm},
Hutz--Ingram~\cite{hutz-ingram}, and Wang--Zhang~\cite[Section 5]{src},
in addition to the theoretical results of
Doyle~\cite{doyle-small, doyle-curves, doyle-cyclotomic}
and Krumm~\cite{krumm}.
\begin{conjecture}
	\label{con:n=5-6}
	Let $K$ be a quadratic number field.
	\begin{enumerate}[(a)]
		\item If $c \in \Q$, then $\phi_{2,c}$
			has no periodic points of exact period $5$ in $K$.
		\item If $c \in \Q - \{-\frac{71}{48}\}$,
			then $\phi_{2,c}$
			has no periodic points of exact period $6$ in $K$.
	\end{enumerate}
\end{conjecture}

Let $C_1(N)$
denote the dynatomic modular curve parametrizing pairs
$(\phi_{d, c}, z_0)$ of unicritical polynomials and
periodic points of period $N$, and let $C_0(N)$
denote the dynatomic modular curve parametrizing pairs
$(\phi_{d, c}, \sco)$
of unicritical polynomials and $N$-cycles. There is a pair
of projections
\[
	\begin{tikzcd}
		C_1(N) \arrow[r] & C_0(N) \arrow[r] & \P^1.
	\end{tikzcd}
\]
In Section~\ref{sec:rational-points},
we show that if $K$ is a quadratic number field
and $(\phi_{d, c}, N, K)$ satisfies the {\gdr},
then the fiber above $c \in \P^1(\Q)$
on $C_1(N)(K)$ maps to a point in $C_0(N)(\Q)$.
Therefore, if the points of $C_0(N)(\Q)$ are completely known
for a fixed $N$
and $(\phi_{d, c}, N, K)$ is known to satisfy the {\gdr} for all
$c \in \Q$, then one would be able to completely determine
the periodic points of $\phi_{d, c}$ in $K$ of period $N$.
In this fashion, the {\gdr} acts as a bridge between
the $\Q$-rational points of $C_0(N)$ and the $K$-rational
points of $C_1(N)$.
Since the points of $C_0(N)(\Q)$ are well-understood for
for $N = 5$ and $N = 6$ due to
Flynn--Poonen--Schaefer~\cite{n=5} and Stoll~\cite{n=6},
the possible counterexamples to Conjecture~\ref{con:n=5-6}
are constrained by the finite sets $\Sigma_{2, 5}$ and
$\Sigma_{2, 6}$.

\begin{corollary}
  \label{cor:n=5-zero-2}
	Let $K$ be a quadratic number field.
	If $c \in \Q - \Sigma_{2, 5}$, then
	$\phi_{2,c}$
	has no periodic points of exact period $5$ in $K$.
\end{corollary}
\begin{corollary}
	\label{cor:n=6-zero-2}
	Let $K$ be a quadratic number field and let
	$J$ be the Jacobian of $C_0(6)$.
	Assume that the $L$-series $L(J, s)$ extends to an entire function,
	$L(J, s)$ satisfies the standard functional equation, and the weak
	Birch and Swinnerton-Dyer conjecture is valid for $J$.
	If $c \in \Q - (\Sigma_{2, 6} \cup \{-\frac{71}{48}\})$, then
	$\phi_{2, c}$ has no periodic points of exact period $6$ in $K$.
\end{corollary}

\begin{remark}
	\textit{A priori},
	the statement of Corollary~\ref{cor:n=5-zero-2} can be recovered
	from the finiteness of the exceptional set $\Sigma_{2, 5}$ from
	Krumm~\cite{krumm-finiteness}, since the $30$ periodic points of
	exact period $5$ must then be conjugate
	(and therefore not quadratic)
	for all but finitely many rational $c$.
	For Corollary~\ref{cor:n=6-zero-2},
	Doyle--Faber--Krumm~\cite[Section~1.4]{doyle-faber-krumm} observed
	that the results of Stoll~\cite{n=6} actually imply that there
	are only finitely many pairs $(K, c)$ of quadratic number fields $K$
	and $c \in K$ such that $\phi_{2, c}$ has a periodic point of exact
	order $6$ in $K$.
\end{remark}

\subsection{Future directions}
If the {\gdr} can be
established for all $c \in \Q$ for given pairs $(d, N)$, then 
future developments in determining the structure of $C_0(N)(\Q)$ and
$C_1(N)(\Q)$ could lead to more general statements of
Poonen's conjecture (cf. \cite{poonen-conj}) for number fields.
It is not known if there are any tuples
$(\phi_{2, c}, N, K/\Q)$
that do not satisfy the {\gdr}, but
we show that $c$ must lie in the density $0$ subset
$\Sigma_{2, N}$ of the rationals for any such tuple.
There are two parallel directions that would directly extend
Theorem~\ref{thm:gdr-d2} and provide more cases of
the {\gdr}: further identifying the exceptional sets $\Sigma_{d, N}$
and direct verification of the {\gdr} for the exceptional cases
$c \in \Sigma_{d, N}$.

In the first direction (for $d=2$),
the program of Krumm~\cite{krumm-galois,
krumm-finiteness} in extending the results of Morton~\cite{morton-1992}
has yielded significant progress in determining the exceptional
sets $\Sigma_{2, N}$ for small $N$.
The finiteness methods developed by
Krumm are constrained by computational limitations
for larger $N$
(cf.~\cite[Section~9]{krumm-finiteness}).
However, further calculation of the exceptional sets
$\Sigma_{2, N}$ (``explicit Hilbert irreducibility'')
would yield immediate extensions of
Theorem~\ref{thm:gdr-d2} for $N > 4$.
In particular, Corollaries~\ref{cor:n=5-zero-2}
and \ref{cor:n=6-zero-2} would be immediately
refined by the explicit determination
of $\Sigma_{2, 5}$ and $\Sigma_{2, 6}$.

In the second direction, the techniques in
Section~\ref{sec:main-proof} for checking
the exceptional cases of $c \in \Sigma_{2, N}$ when
$N \leq 4$ have been ad-hoc and depend on explicit
factorizations of $\Phi_N(z, c)$ in $\Q[z]$.
Any case-by-case approach for the exception cases
may depend on progress in the first direction,
since an explicit description for
the exceptional set $\Sigma_{d, N}$ is not known
if $N > 4$ or $d > 2$. However, it is worth
pointing out that if $(\phi_{2, c}, 5, K/\Q)$ and
$(\phi_{2, c}, 6, K/\Q)$
satisfy the {\gdr} for all $c \in \Sigma_{2, 5}$
and $\Sigma_{2, 6}$ respectively, then Conjecture~\ref{con:n=5-6}
would be affirmatively resolved (assuming the standard conjectures
for $N=6$).

Outside of those two directions, the {\gdr}
can also be considered for broader classes of functions $\phi$.
For instance, the results of Section~\ref{subsec:exceptional-set}
for unicritical polynomials $\phi_{d, c}$
and their consequences
would also extend to any nonlinear polynomial $\phi(z)$ whose
associated dynatomic polynomial $\Phi_N(z, t)$ is known to be
irreducible over $\Q(t)$.
Furthermore, Definition~\ref{def:A} of the {\gdr}
is made in sufficiently generality that this program
can also be considered for dynamical systems over finite fields,
dynamical systems over function fields,
and more general dynamical systems on algebraic varieties.
There are, for instance, interesting calculations done by
by Bridy--Garton~\cite{bridy-garton} and
Krumm--Sutherland~\cite{krumm-sutherland} that could be
used for such analogues.

%%%%%%%%%%%%%%%%%%%%%%%%%%%%%%%%%%%%%%%%%%%%%%%%%%%%%%%%%%%%%%%%%

\section{Irreducibility}
\label{sec:crit}
\subsection{An irreducibility criterion}
\label{subsec:crit}
The dynamical mimicry of the action of the Galois group of
$\Phi_N$ was previously considered by Vivaldi--Hatjispyros for
any nonlinear polynomial $\phi$. They
established that if the dynatomic polynomial $\Phi_N(z)$ is
irreducible as a polynomial in $\Q[z]$ for a given $N$, then 
for each orbit $\O$ there is a nontrivial
subgroup $H_\O$ of the Galois group of $\Phi_N(z)$
that acts on $\O$ in the same way as iterates of
$\phi$ \cite[Section~3]{vh}. As a consequence of
the work of Vivaldi--Hatjispyros and the following general
fact, there is an 
irreducibility criterion for the {\gdr}.
\begin{lemma}
	\label{lem:intermediate-extensions}
	Let $\phi$ be a polynomial in $k[z]$,
	$N$ be an integer greater than $1$, and
	$K'$ be the splitting field of the dynatomic
	polynomial $\Phi_N$ over $k$.
	If $(\phi, N, K'/k)$ satisfies the {\gdr}, then
	$(\phi, N, K/k)$ satisfies the {\gdr} for
	all nontrivial Galois extensions $K/k$.
\end{lemma}

\begin{proof}
	First, we show that $(\phi, N, L/k)$	satisfies the {\gdr} for any
	nontrivial intermediate Galois extension $L$ of $K'/k$.
	If $\per_{N, L-k}(\phi)$ is empty, then $(\phi, N, L/k)$
	vacuously satisfies the {\gdr}, so we may assume that
	$\per_{N, L-k}(\phi)$ is nonempty.
	Let $z_0 \in \per_{N, L-k}(\phi) \subset \per_{N, K'-k}(\phi)$. 
	By the definition of the {\gdr} for $(\phi, N, K'/k, z_0)$,
	there is a positive integer $i < N$ and a nontrivial $\tau' \in
	\gal(K'/k)$ such that $\phi^i(z_0) = \tau' (z_0)$.
	Notice that $\tau'$ cannot be in the subgroup $\gal(K'/L)$,
	since $\tau'$ does not fix $z_0$, which is an element of $K$. Then
	$\tau'$ corresponds to a nontrivial element $\tau$ of the
	quotient group $\gal(L/k)$ such that $\phi^i(z_0) = \tau (z_0)$.
	Therefore $(\phi, N, L/k)$ satisfies the {\gdr} for any nontrivial
	intermediate Galois extension $L/k$ of $K'/k$.
	
	Now consider the arbitrary nontrivial Galois extension $K/k$, for
	which we can similarly assume that $\per_{N, K-k}(\phi)$ is
	nonempty. Let $z_0 \in \per_{N, K-k}(\phi)$.
	Since $z_0$ is a root of $\Phi_N(z)$, it is also
	contained in $\per_{N, K'-k}(\phi)$. For the
	intermediate Galois extension $(K \cap K')/k$, the triple
	$(\phi, N, (K \cap K')/k)$ satisfies the {\gdr} and $z_0 \in
	\per_{N, (K \cap K')-k}(\phi)$.
	Then there is a positive integer $i < N$ and
	a nontrivial $\sigma' \in \gal((K \cap K')/k)$
	such that $\phi^i(z_0) = \sigma' (z_0)$. 
	By taking a lift $\sigma \in \gal(K/k)$ of $\sigma'$,
	$(\phi, N, K/k)$ satisfies the {\gdr}.
\end{proof}

\begin{proposition}[Irreducibility criterion]
\label{prop:vh2}
	Let $\phi$ be any	nonlinear polynomial in $\Q[z]$,
	$N$ be an integer greater than $1$, and $K/\Q$ be
	a nontrivial Galois number field.
	If the dynatomic polynomial $\Phi_N(z)$ is
	irreducible in $\Q[z]$, then
	$(\phi, N, K/\Q)$ satisfies the {\gdr}.
\end{proposition}
\begin{proof}
	Let $K'$ be the splitting field of $\Phi_N$ over $\Q$.
	Vivaldi--Hatjispyros~\cite[Section~3]{vh} proved that if
	$\Phi_N(z)$ is irreducible in $\Q[z]$, then
	for each orbit $\O$ there is
	a nontrivial subgroup $H_\O$ of the Galois group of $\Phi_N(z)$
	that acts on $\O$ in the same way as iterates of
	$\phi$. Consider any periodic point $z_0 \in \per_{N, K'-\Q}(\phi)$
	and its orbit $\O$. Then
	there is a nontrivial $\sigma \in H_\O \subset \gal(K'/\Q)$
	such that $\sigma(z) = \phi^i(z)$ for some positive integer $i < N$,
	so $(\phi, N, K'/\Q)$ satisfies the {\gdr}.
	By Lemma~\ref{lem:intermediate-extensions},
	$(\phi, N, K/\Q)$ satisfies the {\gdr} as well.
\end{proof}

\begin{remark}
	The notion of an automorphism polynomial, which was defined by the
	contemporaneous study of Morton--Patel~\cite{morton-patel}
	for non-constant polynomials $\phi \in k[z]$, is also closely
	related to the {\gdr} when $\Phi_N(z)$ is irreducible in $k[z]$.
	If $\Phi_N(z)$ is irreducible in $k[z]$ and
	$\phi$ is an automorphism polynomial of $\Phi_N(z)$,
	then for $K'$ the splitting	field of $\Phi_N$ over $k$,
	$(\phi, N, K'/k)$ satisfies the {\gdr}
	with $i = 1$ for each periodic point $z_0$ of exact period $N$.
	By Lemma~\ref{lem:intermediate-extensions}, this is actually sufficient
	for	$(\phi, N, K/k)$ to satisfy the {\gdr} for any nontrivial
	Galois extension $K/k$.
\end{remark}

\subsection{Hilbert irreducibility and the exceptional set for unicritical polynomials}
\label{subsec:exceptional-set}
For unicritical polynomials $\phi_{d, c}$,
Vivaldi--Hatjispyros~\cite[Section~3]{vh} remarked that the dynatomic
polynomial $\Phi_N(z, c)$ being irreducible in $\Q[z]$ ``appears to be
the typical situation'',
meaning that it is reducible for $c$ in a density $0$ subset of $\Q$.
They proved their hypothesis for $d = 2$ and $N \leq 3$, but
nevertheless $\Phi_N(z, c)$ is not irreducible in $\Q[z]$ in
general. Even for $d = 2$,
Vivaldi--Hatjispyros demonstrated that
$\Phi_N(z, c)$ is never irreducible in $\Q[z]$
for an infinite family of $c$ when $N = 3$, for $c = -2$ when
$N \geq 3$, and for $c = 0$ when $2^N - 1$ is not a Mersenne prime.

To study the irreducibility of $\Phi_N(z, c) \in \Q[z]$ for
$c \in \Q$, we can view it as a specialization of $\Phi_N(z, t)
\in \Q[z, t]$. By the later results of
Bousch~\cite[Chapitre 3, Th\'{e}or\`{e}me 1]{bousch}
and Lau--Schleicher~\cite[Theorem~4.1]{lau-schleicher}
(cf. Morton~\cite[Corollary~1]{morton-irred},
Buff--Lei~\cite[Theorem~1.2]{buff-lei},
Gao--Ou~\cite[Theorem~1.2]{gao-ou}), it is known that 
$\Phi_N(z, t)$ is irreducible as a polynomial in $\C(t)[z]$.
Let $G_N$ denote the Galois group of $\Phi_N(z, t)$ over $\Q(t)$
and let $G_{N,c}$ denote the Galois group of its specialization
$\Phi_N(z, c)$ over $\Q$. Define $\Sigma_{d, N}$ to be the
locus of $c \in \P^1(\Q)$ such that $G_N \cong G_{N,c}$.
The thinness (in the sense of Serre) of
$\Sigma_{d, N}$ follows
immediately from an application of Hilbert's irreducibility theorem.
\begin{corollary}
	\label{cor:hilbert}
	For each integer $d$ and $N$ greater than $1$,
	$\Sigma_{d, N}$ is a thin subset (in the sense of Serre)
	of $\P^1(\Q)$ such
	that for all rational $c$ not in $\Sigma_{d, N}$,
	\begin{enumerate}[(a)]
		\item	$\Phi_N(z, c)$ is irreducible in $\Q[z]$, and
		\item $G_N \cong G_{N,c}$.
	\end{enumerate}
\end{corollary}

\begin{remark}
	Thin sets (in the sense of Serre) in $\P^1(\Q)$ have density $0$,
	as the number of points of a thin set in $\P^1(\Q)$ with naive height
	at most $H$ is $\mathrm{O}(H)$
	(cf. Serre~\cite[Proposition~3.4.2]{serre-galois}).
	Corollary~\ref{cor:hilbert} implies that, for all $N$,
	the set of rational $c$
	such that $\Phi_N(z, c)$ is reducible in $\Q[z]$ is a density $0$
	subset of $\Q$.
\end{remark}

When $\Phi_N(z, c)$ is irreducible in $\Q[z]$,
the Galois group $G_{N, c}$ acts transitively on the periodic
points of $\phi_{d, c}$ of period $N$.
We obtain the following useful statement as a
consequence of Proposition~\ref{prop:vh2} and
Corollary~\ref{cor:hilbert}.

\begin{proposition}
	\label{prop:galois-almost-2}
	For each integer $d$ and $N$ greater than $1$, there exists a
	thin set (in the sense of Serre) $\Sigma_{d, N}$ in $\P^1(\Q)$ such
	that for all rational $c$ not in $\Sigma_{d, N}$,
	\begin{enumerate}[(a)]
		\item the triple $(\phi_{d, c}, N, K/\Q)$
			satisfies the {\gdr} for all nontrivial Galois extensions
			$K/\Q$, and
		\item $G_N \cong G_{N,c}$.
	\end{enumerate}
\end{proposition}

\begin{remark}
	In fact, Corollary~\ref{cor:hilbert} (and therefore
	Proposition~\ref{prop:galois-almost-2}) is true for a larger class
	of polynomials than the unicritical polynomials $\phi_{d, c}$.
	By Morton~\cite[Theorem~B]{morton-irred}
	(cf. Morton~\cite[Theorem~B]{morton-galois}), the associated
	dynatomic polynomial $\Phi_N(z, t)$ is also irreducible in $\C(t)[z]$
	for polynomials $\phi \in \Z[z, t]$ of degree $d \geq 2$
	satisfying:
	\begin{enumerate}[(a)]
		\item $\phi(z, t^i)$ is homogeneous for some integer $i \geq 1$,
		\item $\phi(z, 0) = z^d$,
		\item $\phi(z, 1)$ has distinct roots, and
		\item the primitive $N$-bifurcation points of $\phi$ are distinct.
	\end{enumerate}
	This class of polynomials includes
	$\phi_{d, c} = z^d + c$ (by setting $c = -t^d$).
\end{remark}

Understanding when
$G_N \not\cong G_{N,c}$ determines the
exceptional set $\Sigma_{d, N}$.
On the one hand, the structure of the Galois group
$G_N$ of $\Phi_N(z, t)$ over $\Q(t)$ is well-understood.
Bousch~\cite[Chapitre 3]{bousch}
(c.f. \cite[Theorem~4.2]{morton-patel}) showed that the Galois group
$G_N$ of $\Phi_N(z, t)$ over $\Q(t)$ is isomorphic to a wreath
product
\[
	G_N \cong (\Z/N\Z) \wr S_r,
\]
\noindent
where $r$ is an integer such
that $rN = \deg \Phi_N$, for $\phi_{2,c}$. This result was extended
by Lau--Schleicher~\cite{lau-schleicher}
(cf. Morton~\cite{morton-galois})
for $\phi_{d, c}$ to all $d \geq 2$.

However, the structure of the Galois group
$G_{N,c}$ of the specialization $\Phi_N(z, c)$ over $\Q$ is not known
in general. For $d = 2$, one can show that $G_{N, c}$ is not
isomorphic to $G_N$ for any integer $N > 1$ if
$c = 0$ or $-2$ (i.e. that $0, -2 \in \Sigma_{2,N}$ for all $N > 1$)
since the Galois group $G_{N,c}$ is abelian for those
values of $c$.
Otherwise, there are only explicit descriptions of
$\Sigma_{2, N}$ for small $N$
due to Morton~\cite{morton-1992} and
Krumm~\cite{krumm-galois, krumm-finiteness}

%%%%%%%%%%%%%%%%%%%%%%%%%%%%%%%%%%%%%%%%%%%%%%%%%%%%%%%%%%%%%%%

\section{Quadratic points on dynatomic curves}
\label{sec:rational-points}

The study of periodic points of exact period $N$ in a number field
$K$ is related to the study of
$K$-points on the curves $C_1(N)$ and $C_0(N)$:
a pair $(z_0, \phi_{d, c})$ of
a map $\phi_{d, c}(z) = z^d + c$ with $c \in K$
and a periodic point $z_0 \in K$ of exact period $N$ of
$\phi_{d, c}$ gives rise to a $K$-point on $C_1(N)$,
and 
a pair $(\O, \phi_{d,c})$ of a map $\phi_{d,c}(z) =
z^d + c$ with $c \in K$ and a $\gal(\overline{K}/K)$-stable
$N$-cycle $\O$ gives rise to a $K$-point on $C_0(N)$.
The $K$-points on these curves contain
full information about periodic points in $K$.

The first observation about the {\gdr} is that we can specify
the dynamical indices corresponding to the Galois action.
The following elementary lemma is true
in greater generality for any rational map $\phi$ and
any nontrivial finite Galois extension $K/k$, but we state
it in the unicritical polynomial and number field setting.

\begin{lemma}
  \label{lem:galois}
	Let $K$ be a nontrivial finite Galois extension of $\Q$ of	
	degree $D$, let $d$ and $N$ be integers greater than $1$,
	let $c \in \Q$, and denote $g := \gcd(N, D)$.
	Let $z_0 \in K - \Q$ be a periodic point
	of $\phi_{d,c}$ with exact period $N$.
	
	If there exists a nontrivial $\t \in \gal(K/\Q)$ and a
	positive integer $i < N$ such that
	\[
		\t z_0 = \phi_{d,c}^i(z_0),
	\]
	then $g > 1$ and $i = \frac{mN}{g}$ for some integer $m$
	such that $1 \leq m \leq g - 1$.
\end{lemma}

\begin{proof}
	Suppose that there exists a nontrivial $\t \in \gal(K/\Q)$ such that
	$\phi_{d, c}^i(z_0)= \t(z_0)$ for some $i \in
	\{1, \ldots, N - 1\}$. Since
	$\t$ commutes with $\phi_{d, c}$ and the action of $\phi_{d, c}$
	is transitive on the $N$-cycle, we have that $\t \equiv
	\phi_{d, c}^i$ on the entire cycle.

  Since $K$ is Galois, the order of the Galois group
  $\gal(K/\Q)$ is $D$. Thus,
  \[
		z_0 = \t^D(z_0) = (\phi_{d, c}^i)^D(z_0) = \phi_{d, c}^{iD}(z_0).
  \]
	Since $z_0$ has exact period $N$, it follows that
	$N$ divides $iD$ and $i$ is a multiple of $\frac{N}{g}$.
	In particular, there is a contradiction if $g = 1$
	since $0 < i < N$ by assumption.
\end{proof}

\begin{remark}
	In other words, $\phi_{d, c}$ satisfies the {\gdr}
	for $N$, $K$, and $z_0 \in K - \Q$ if and only if there exists a
	nontrivial $\t \in \gal(K/\Q)$ and $m \in \{1, \ldots, g - 1\}$
	such that $\phi_{d,c}^{\frac{mN}{g}}(z_0) = \t z_0$ (this fact
	becomes convenient for showing Lemma~\ref{lem:n=4}).
	In particular,
	a necessary condition for $(\phi_{d, c}, N, K/\Q, z_0)$ to satisfy
	the {\gdr} is that $N$ and $[K:\Q]$ are not coprime.
\end{remark}

If $K$ is a quadratic number field, then we can use
Lemma~\ref{lem:galois} to make the following
observation.

\begin{lemma}
\label{lem:rational-points}
	Let $K$ be a quadratic number field, $c$ be a rational number,
	and $d$ and $N$ be integers greater
	than $1$. If $(\phi_{d, c}, N, K/\Q)$ satisfies the
	{\gdr}, then:
	
	\begin{enumerate}[(a)]
		\item For any $N$-cycle $\{z_0, \ldots z_{N-1}\}$ of $\phi_{d, c}$
			in $K - \Q$, its trace $\sum z_i$ is rational; \\
		\item Furthermore, each $N$-periodic point $z_i$ is rational
			if $N$ is odd.
	\end{enumerate}
\end{lemma}
\begin{proof}
	(a): If $N$ is even, the $N$-cycle consists of pairs of Galois
	conjugates $(z_i, z_{i + \frac{N}{2}})$ by Lemma~\ref{lem:galois}.
	Then the
	trace $\sum_{i = 0}^{N-1} z_i$ must
	lie in $\Q$.

	(b): If $N$ is \textit{odd}, then the set of periodic
	points $\per_{N, K-\Q}(\phi_{d, c})$ must be empty because
	otherwise, $i \in \{1, \ldots, N - 1\}$ would be a multiple of $N$ by
	virtue of the fact that $g = \gcd(N, 2) = 1$.
	Hence any periodic point of $\phi_{d, c}$ of odd exact period $N$ 
	contained in $K$ must itself be rational.
\end{proof}

\begin{remark}
	Effectively, Lemma~\ref{lem:rational-points} is the observation
	that when $K$ is a quadratic extension, the {\gdr}
	says that the nontrivial automorphism of $K$ setwise fixes
	$N$-cycles in $K$. This forces the trace to be rational and induces
	an involution on each $N$-cycle in $K$, which must have a fixed
	point when $N$ is odd.
\end{remark}

\begin{remark}
	Lemma~\ref{lem:rational-points}(a) can be restated as saying that
	if $c \in \Q$ and $(\phi_{d, c}, N, K)$ satisfies the {\gdr},
	then the fiber above $c$ on $C_1(N)(K)$ maps to a point
	in $C_0(N)(\Q)$. Lemma~\ref{lem:rational-points}(b) can be
	restated as saying that if furthermore $N$ is odd, then
	the fiber above $c$ on $C_1(N)(K)$ is actually contained in
	$C_1(N)(\Q)$.	
\end{remark}

Combined with existing results on the rational points of
dynatomic curves, Lemma~\ref{lem:rational-points} immediately yields
the connection between the {\gdr} and Conjecture~\ref{con:n=5-6}.

\begin{corollary}
  \label{cor:n=5-zero}
	Let $K$ be a quadratic number field and $c$ be a rational number. If
	$(\phi_{2, c}, 5, K)$ satisfies the {\gdr}, then $\phi_{2, c}$ has
	no periodic points of exact period $5$ in $K$.
\end{corollary}

\begin{proof}
	The parameter $c$ and the trace of an $N$-cycle give coordinates
	on $C_0(N)$.
	Since we assume that $c$ is a rational number, it is
	sufficient to show that the trace of an $N$-cycle is rational
	to demonstrate that an $N$-cycle corresponds to a point in
	$C_0(N)(\Q)$.
	
	By Lemma~\ref{lem:rational-points}, a $5$-cycle of
	$\phi_{2,c}$ necessarily has rational trace and
	corresponds to a $\Q$-point on $C_0(5)$.
	By the work of
	Flynn--Poonen--Schaefer~\cite{n=5}, the genus $2$
	dynatomic curve $C_0(5)$ has only six rational
	points, three of which correspond to $c = \infty$ and three of which
	ultimately correspond to $5$-cycles in quintic cyclic extensions of
	$\Q$.
\end{proof}

\begin{corollary}
  \label{cor:n=6-zero}
	Let $J$ be the Jacobian of $C_0(6)$. Suppose that
	the $L$-series $L(J, s)$ extends to an entire function,
	$L(J, s)$ satisfies the standard functional equation, and the weak
	Birch and Swinnerton-Dyer conjecture is valid for $J$.
	
	Let $K$ be a quadratic number field and $c$ be a rational number. If
	$(\phi_{2, c}, 6, K)$ satisfies the {\gdr}, then $\phi_{2, c}$ has
	no periodic points of exact period $6$ in $K$,
	unless $K = \Q(\sqrt{33})$ and
	$c = -\frac{71}{48}$ in which case there is exactly one $6$-cycle:
	\[
  \label{eq:6-cycle}
		z_0 = -1 + \frac{\sqrt{33}}{12},\,
		z_1 = -\frac{1}{4} - \frac{\sqrt{33}}{6},\,
		z_2 = -\frac{1}{2} + \frac{\sqrt{33}}{12},\,
		z_{i+3} = \sigma (z_{i}),
	\]
	where $\sigma$ is the nontrivial element of $\gal(K/\Q)$.
\end{corollary}
\begin{proof}
	By Lemma~\ref{lem:rational-points}, the trace of a
	$6$-cycle of $\phi_{2,c}$ in a quadratic number field $K$ is
	rational so it would
	necessarily correspond to a point in $C_0(6)(\Q)$.
	By the work of Stoll~\cite{n=6}, which is dependent on the
	usual conjectures for $L(J, s)$, there are only ten points
	in $C_0(6)(\Q)$. Five of the points correspond to cusps on
	$C_1(6)$ and the other five points correspond to
	explicitly-described $6$-cycles; the
	following cycle with $c = -\frac{71}{48}$ and $K = \Q(\sqrt{33})$
	is the only one defined over a quadratic field:
	\[
		z_0 = -1 + \frac{\sqrt{33}}{12},\,
		z_1 = -\frac{1}{4} - \frac{\sqrt{33}}{6},\,
		z_2 = -\frac{1}{2} + \frac{\sqrt{33}}{12},\,
		z_{i+3} = \sigma (z_{i}),
	\]
	where $\sigma$ is the nontrivial element of $\gal(K/\Q)$.
\end{proof}

\begin{remark}
	In the situations of Corollaries~\ref{cor:n=5-zero} and
	\ref{cor:n=6-zero}, the irreducibility criterion of
	Proposition~\ref{prop:vh} is not sufficient. For example,
	$\Phi_6(z, c)$ is reducible as a
	polynomial in $\Q[z]$ when $c = -2$.
\end{remark}

\begin{remark}
	Corollaries~\ref{cor:n=5-zero-2} and \ref{cor:n=6-zero-2}
	follow immediately from the application of Theorem~\ref{thm:gdr-d2}
	to Corollaries~\ref{cor:n=5-zero} and \ref{cor:n=6-zero}
\end{remark}

%%%%%%%%%%%%%%%%%%%%%%%%%%%%%%%%%%%%%%%%%%%%%%%%%%%%%%%

\section{The proof of Theorem~\ref{thm:gdr-d2}}
\label{sec:main-proof}

In Theorem~\ref{thm:gdr-d2}, we want to show
that the quadratic polynomial $(\phi_{2, c}, N, K/\Q)$
satisfies the {\gdr} for all nontrivial Galois extensions
$K/\Q$ with various $N \leq 9$ and $c \in \Q$.
For the proof, we 
proceed case-by-case for varying $N$.
For each $N$, we first use Proposition~\ref{prop:galois-almost-2}
and descriptions of $\Sigma_{2, N}$ to show that
$(\phi_{2, c}, N, K/\Q)$ satisfies the {\gdr}
when $c \in \Q - \Sigma_{2, N}$, before then
directly checking the $c \in \Sigma_{2, N}$.

\subsection{\texorpdfstring{Period $2$}{Period 2}}
\begin{lemma}
	\label{lem:n=2}
	The tuple $(\phi_{2, c}, 2, K/\Q)$ satisfies the
	{\gdr} for all nontrivial Galois number fields $K/\Q$ and all
	$c \in \Q$.
\end{lemma}
\begin{proof}
	For $N = 2$, there can only be at most one $2$-cycle of
	$\phi_{2, c}$ for a given $c \in \Q$ since $\Phi_2(z, c)$ is a
	degree $2$ polynomial in $\Q[z]$. Therefore either $\Phi_2(z, c)$ is
	irreducible or the periodic points are rational.
\end{proof}

\subsection{\texorpdfstring{Period $3$}{Period 3}}

\begin{lemma}
	\label{lem:n=3}
	The tuple $(\phi_{2, c}, 3, K/\Q)$ satisfies the
	{\gdr} for all nontrivial Galois number fields $K/\Q$ and all
	$c \in \Q$.
\end{lemma}

\begin{proof}
	The dynatomic polynomial $\Phi_3(z, c)
	\in \Q[z]$ for a given $c \in \Q$
	can either be irreducible or factor into two (possibly reducible)
	cubic factors
	(cf. Vivaldi--Hatjispyros~\cite[Section~2]{vh}
	and Morton~\cite[Theorem~3]{morton-1992}).
	By Proposition~\ref{prop:vh}, $(\phi_{2, c}, 3, K/\Q)$ satisfies the
	{\gdr} for any Galois number field $K/\Q$ if $\Phi_3(z, c)$ is
	irreducible in $\Q[z]$, so we can assume that
	$\Phi_3(z, c)$ factors into two cubic factors.

	Any linear factor of $\Phi_3(z, c)$ corresponds to a periodic
	point in $\Q$.
	A $3$-cycle of $\phi_{2, c}$ containing a rational number must
	necessarily be entirely contained in $\Q$, so
	$\phi_{2, c}$ has a $3$-cycle entirely
	contained in $\Q$. We can disregard such a $3$-cycle,
	since the {\gdr} only concerns irrational periodic points.
	
	For any irreducible cubic factor of $\Phi_3(z, c)$,
	two of its roots must lie in the same $3$-cycle.
	Then $\phi_{2, c}$ satisfies the
	{\gdr} for that $3$-cycle and any Galois number field $K/\Q$
	containing it since the $3$-cycle is not disjoint
	from its $\sigma$-conjugates for some nontrivial
	$\sigma \in \gal(K/\Q)$. Alternatively, observe that
	if $\phi_{2, c}(z) = \sigma (z)$ then
	\[
		\phi_{2, c}^2(z) = \phi_{2, c}(\sigma (z)) = \sigma \phi_{2, c}(z)
		= \sigma^2 (z).
	\]
\end{proof}

\begin{remark}
	Walde--Russo \cite[Corollary~2]{walde-russo} and
	Vivaldi--Hatjispyros \cite[Section~5]{vh} showed that the Galois
	group of any irreducible cubic factor of $\Phi_3(z, c)$ is 
	isomorphic to $A_3$.	
\end{remark}

\begin{remark}
	Morton~\cite[Theorem 8]{morton-1992} showed that
	\[
		\Sigma_{2,3} = \set{0, -\frac{7}{2}} \cup
			\set{-\frac{r^3 + 29r^2 + 243r + 559}
			{16(r + 7)(r + 11)} \middle| r \in \Q - \set{-7, -11}} \cup
			\set{-\frac{s^2 + 7}{4} \middle| s \in \Q}.
	\]
	In fact, Morton showed that while $G_{N, c}$ is not isomorphic
	to the full wreath product when
	$c = 0, -\frac{7}{2}$, or $-\frac{r^3 + 29r^2 + 243r + 559}
	{16(r + 7)(r + 11)}$ for any $r \in \Q - \set{-7, -11}$,
	the dynatomic polynomial $\Phi_3(z, c)$ is
	still irreducible in these cases.
	This demonstrates that $G_{N, c} \cong G_N$
	is not a necessary condition for
	$(\phi_{2, c}, 3, K/\Q)$ to satisfy the {\gdr}.
	In fact, we did not need to know anything about $\Sigma_{2, 3}$
	to prove Lemma~\ref{lem:n=3}
\end{remark}

\subsection{\texorpdfstring{Period $4$}{Period 4}}

\begin{lemma}
\label{lem:n=4}
	The tuple $(\phi_{2, c}, 4, K/\Q)$ satisfies the
	{\gdr} for Galois number fields $K$ and $c \in \Q$
	in the following cases:
	\begin{enumerate}[(a)]
		\item if $[K:\Q] = 2$, or
		\item if $[K:\Q] > 2$ and $c \notin
			\set{\frac{-s^3 - 2s + 4}{4s} \middle| s \in \Q^\times}$.
	\end{enumerate}
\end{lemma}

\begin{proof}
	(a): There is an explicit complex parametrization of $4$-cycles
	of $\phi_{2, c}$ due to Netto~\cite{netto},
	Morton~\cite{n=4}, and
	Erkama~\cite{erkama}. 
	Panraksa~\cite[Theorem~1.5.1.]{panraksa} used this parametrization
	to study the
	factorization of $\Phi_4(z, c)$ and prove that for a quadratic
	number field $K$, $c \in \Q$, and a $4$-cycle
	$\{z_0, z_1, z_2, z_3\} \subset K$ of $\phi_{2, c}$,
	$z_0$ and $z_2$ are Galois conjugates. Therefore,
	$(\phi_{2, c}, 4, K/\Q, z_0)$ satisfies the 
	{\gdr}. But $z_0$ is an arbitrary member of an arbitrary $4$-cycle,
	so $(\phi_{2, c}, 4, K/\Q)$ satisfies the {\gdr}.
	
	(b): Krumm~\cite[Proposition~4.8]{krumm-galois} showed that
	\[
		\Sigma_{2,4} = \set{-\frac{5}{2}} \cup
			\set{\frac{s^2 + 2s - 4}{8s} \middle| s \in \Q^\times} \cup
			\set{\frac{-s^3 - 3s + 4}{4s} \middle| s \in \Q^\times},
	\]
	with explicit presentations of the exceptional Galois groups as
	subgroups of $S_{12}$. In fact, $\Phi_4(z, c)$ is still irreducible
	when $c \in \set{-\frac{5}{2}} \cup
	\set{\frac{s^2 + 2s - 4}{8s} \middle| s \in \Q^\times}$, even though
	it generates a different Galois group.
	By Proposition~\ref{prop:vh}, $(\phi_{2, c}, 4, K/\Q)$ still
	satisfies the {\gdr} in those cases.
\end{proof}

\begin{remark}
	The factorization types of $\Phi_4(z, c)$ with $c \in \Q$
	are explicitly known
	due to Morton~\cite[Theorem~4]{n=4},
	Panraksa~\cite[Theorem~2.3.5.]{panraksa},
	and Krumm~\cite[Theorem 1.2]{krumm-galois}.
	In particular, if
	$c \in \set{\frac{-s^3 - 3s + 4}{4s} \middle| s \in \Q^\times}$
	then $\Phi_4(z, c)$ factors either into
	an irreducible degree $8$ polynomial and an irreducible degree $4$
	polynomial, or an irreducible degree $8$ polynomial and two
	irreducible quadratic polynomials.
\end{remark}

\subsection{\texorpdfstring{Period $5$, $6$, $7$, and $9$}{Period 5, 6, 7, and 9}}

\begin{lemma}
\label{lem:n-large}
	The tuple $(\phi_{2, c}, N, K/\Q)$ satisfies the
	{\gdr} for all nontrivial Galois number fields $K$ and for all but
	finitely many $c \in \Q$.
\end{lemma}

\begin{proof}
	For $N \in \{5, 6, 7, 9\}$, the exceptional set $\Sigma_{2,N}$
	is finite due to Krumm~\cite{krumm-finiteness}.
	Since $\Phi_N(z, c)$ is necessarily irreducible for $c \notin
	\Sigma_{2, N}$, we are done by Proposition~\ref{prop:vh}.
\end{proof}

\begin{remark}
	While an explicit description of $\Sigma_{2,N}$ is not
	yet known, Krumm~\cite[Section~9]{krumm-finiteness}
	provides an algorithm for computing the elements of
	$\Sigma_{2,N}$ of bounded naive height.
\end{remark}

\begin{remark}
	For $d = 2$ and $N \geq 5$, it is not known whether any class of
	pairs $(\phi_{2, c}, N)$ satisfies the {\gdr} besides
	those satisfying the irreducibility criteria
	and those explicitly found by
	computational searches for periodic points of exact period $5$ and
	$6$ \cite{n=5, hutz-ingram}. Even when $K/\Q$ is required to be
	quadratic, there
	is an active field of research on 
	understanding the structure of the dynamics of $\phi_{2, c}$
	(cf. \cite{doyle-small, doyle-curves, doyle-cyclotomic, doyle-faber-krumm, hutz-ingram, krumm}).
\end{remark}

%%%%%%%%%%%%%%%%%%%%%%%%%%%%%%%%%%%%%%%%%%%%%%%%%%%%%%%%%%%%%

\section{Acknowledgments}
The author would like to express thanks to Niccol\`{o} Ronchetti
for introducing him to the field of arithmetic dynamics
and to Zhiming Wang for the many preliminary conversations and
projects that helped shape this project. The author is also grateful
to Brian Conrad, Xander Faber, and John Doyle for
numerous suggestions and clarifications. The author is also
thankful to the anonymous referees for pointing out
mistakes in earlier versions of this paper and for
several insightful comments.

This work is based on a project that began under the
support of the Office of the Vice
Provost for Undergraduate Education of Stanford University and
was later continued under the support of the
National Science Foundation Graduate Research
Fellowship Program under Grant No. DGE-1644869. Any opinions,
findings, and conclusions or recommendations expressed in this
material are those of the author and do not necessarily reflect 
views of the National Science Foundation.

%\nocite{*}
\bibliography{bibliography}{}

\begin{thebibliography}{{Mor}96}

\bibitem[BG17]{bridy-garton}
Andrew Bridy and Derek Garton.
\newblock Dynamically distinguishing polynomials.
\newblock {\em Res. Math. Sci.}, 4:Paper No. 13, 17, 2017.

\bibitem[BL14]{buff-lei}
Xavier Buff and Tan Lei.
\newblock The quadratic dynatomic curves are smooth and irreducible.
\newblock In {\em Frontiers in complex dynamics}, volume~51 of {\em Princeton
  Math. Ser.}, pages 49--72. Princeton Univ. Press, Princeton, NJ, 2014.

\bibitem[Bou92]{bousch}
T.~Bousch.
\newblock {\em Sur quelques probl{\`{e}}mes de dynamique holomorphe}.
\newblock PhD thesis, Universit{\'{e}} de Paris-Sud, Centre d'Orsay, 1992.

\bibitem[DFK14]{doyle-faber-krumm}
John~R. Doyle, Xander Faber, and David Krumm.
\newblock Preperiodic points for quadratic polynomials over quadratic fields.
\newblock {\em New York J. Math.}, 20:507--605, 2014.

\bibitem[Doy18]{doyle-small}
John~R. Doyle.
\newblock Preperiodic points for quadratic polynomials with small cycles over
  quadratic fields.
\newblock {\em Math. Z.}, 289(1-2):729--786, 2018.

\bibitem[Doy19]{doyle-curves}
John~R. Doyle.
\newblock Dynamical modular curves for quadratic polynomial maps.
\newblock {\em Trans. Amer. Math. Soc.}, 371(8):5655--5685, 2019.

\bibitem[Doy20]{doyle-cyclotomic}
John~R. Doyle.
\newblock Preperiodic points for quadratic polynomials over cyclotomic
  quadratic fields.
\newblock {\em Acta Arith.}, 196(3):219--268, 2020.

\bibitem[Erk06]{erkama}
Timo Erkama.
\newblock Periodic orbits of quadratic polynomials.
\newblock {\em Bulletin of the London Mathematical Society}, 38:804--814, 2006.

\bibitem[FPS97]{n=5}
E.~V. Flynn, Bjorn Poonen, and Edward~F. Schaefer.
\newblock Cycles of quadratic polynomials and rational points on a genus-{$2$}
  curve.
\newblock {\em Duke Math. J.}, 90(3):435--463, 1997.

\bibitem[GO14]{gao-ou}
Yan Gao and YaFei Ou.
\newblock {The dynatomic periodic curves for polynomial $z \leftrightarrow z^d
  + c$ are smooth and irreducible}.
\newblock {\em Science China Mathematics}, 57(6):1175--1192, 2014.

\bibitem[HI13]{hutz-ingram}
Benjamin Hutz and Patrick Ingram.
\newblock On {P}oonen's conjecture concerning rational preperiodic points of
  quadratic maps.
\newblock {\em Rocky Mountain J. Math.}, 43(1):193--204, 2013.

\bibitem[Hin88]{hindry}
Marc Hindry.
\newblock Autour d'une conjecture de {S}erge {L}ang.
\newblock {\em Invent. Math.}, 94(3):575--603, 1988.

\bibitem[Kru16]{krumm}
David Krumm.
\newblock A {local–-global} principle in the dynamics of quadratic
  polynomials.
\newblock {\em International Journal of Number Theory}, 12(8):2265--2297, 2016.

\bibitem[Kru18]{krumm-galois}
David Krumm.
\newblock Galois groups in a family of dynatomic polynomials.
\newblock {\em Journal of Number Theory}, 187:469 -- 511, 2018.

\bibitem[Kru19]{krumm-finiteness}
David Krumm.
\newblock A finiteness theorem for specializations of dynatomic polynomials.
\newblock {\em Algebra Number Theory}, 13(4):963--993, 2019.

\bibitem[KS21]{krumm-sutherland}
David Krumm and Nicole Sutherland.
\newblock Galois groups over rational function fields and explicit {H}ilbert
  irreducibility.
\newblock {\em J. Symbolic Comput.}, 103:108--126, 2021.

\bibitem[Lan65]{lang-division}
Serge Lang.
\newblock Division points on curves.
\newblock {\em Ann. Mat. Pura Appl. (4)}, 70:229--234, 1965.

\bibitem[LS94]{lau-schleicher}
Eike Lau and Dierk Schleicher.
\newblock Internal addresses in the {M}andelbrot set and {G}alois groups of
  polynomials.
\newblock Preprint, Institute for Mathematical Sciences, Stony Brook, {\#}19,
  1994.

\bibitem[McQ95]{mcquillan}
Michael McQuillan.
\newblock Division points on semi-abelian varieties.
\newblock {\em Invent. Math.}, 120(1):143--159, 1995.

\bibitem[Mor92]{morton-1992}
Patrick Morton.
\newblock Arithmetic properties of periodic points of quadratic maps.
\newblock {\em Acta Arithmetica}, 62(4):343--372, 1992.

\bibitem[{Mor}96]{morton-irred}
Patrick {Morton}.
\newblock {On certain algebraic curves related to polynomial maps.}
\newblock {\em {Compos. Math.}}, 103(3):319--350, 1996.

\bibitem[Mor98a]{n=4}
Patrick Morton.
\newblock Arithmetic properties of periodic points of quadratic maps. {II}.
\newblock {\em Acta Arith.}, 87(2):89--102, 1998.

\bibitem[Mor98b]{morton-galois}
Patrick Morton.
\newblock Galois groups of periodic points.
\newblock {\em Journal of Algebra}, 201(2):401 -- 428, 1998.

\bibitem[MP94]{morton-patel}
Patrick Morton and Pratiksha Patel.
\newblock {The Galois Theory of Periodic Points of Polynomial Maps}.
\newblock {\em Proceedings of the London Mathematical Society}, 68(2):225--263,
  1994.

\bibitem[Net00]{netto}
Eugen Netto.
\newblock {\em Vorlesungen {\"u}ber Algebra II}.
\newblock Teubner, Leipzig, 1900.

\bibitem[Pan11]{panraksa}
Chatchawan Panraksa.
\newblock {\em Arithmetic dynamics of quadratic polynomials and dynamical
  units}.
\newblock PhD thesis, University of Maryland, College Park, 2011.

\bibitem[Poo98]{poonen-conj}
Bjorn Poonen.
\newblock The classification of rational preperiodic points of quadratic
  polynomials over {${\bf Q}$}: a refined conjecture.
\newblock {\em Math. Z.}, 228(1):11--29, 1998.

\bibitem[Poo99]{poonen-mordell-lang}
Bjorn Poonen.
\newblock Mordell--{L}ang plus {B}ogomolov.
\newblock {\em Invent. Math.}, 137(2):413--425, 1999.

\bibitem[Ser00]{serre-cours}
Jean-Pierre Serre.
\newblock {\em {\OE}uvres {C}ollected papers}, volume~IV, chapter
  R{\'{e}}sum{\'{e}} des cours de 1985--1986, pages 33--38.
\newblock Springer-Verlag, Berlin, 2000.
\newblock 1985--1998.

\bibitem[Ser08]{serre-galois}
J.P. Serre.
\newblock {\em Topics in Galois theory}.
\newblock Research Notes in Mathematics. A K Peters, 2008.

\bibitem[Sto08]{n=6}
Michael Stoll.
\newblock Rational 6-cycles under iteration of quadratic polynomials.
\newblock {\em LMS J. Comput. Math.}, 11:367--380, 2008.

\bibitem[VH92]{vh}
F~Vivaldi and S~Hatjispyros.
\newblock Galois theory of periodic orbits of rational maps.
\newblock {\em Nonlinearity}, 5(4):961, 1992.

\bibitem[WR94]{walde-russo}
Ralph Walde and Paula Russo.
\newblock Rational periodic points of the quadratic function {$Q_c(x) = x^2 +
  c$}.
\newblock {\em Amer. Math. Monthly}, 101(4):318--331, 1994.

\bibitem[WZ15]{src}
Zhiming Wang and Robin Zhang.
\newblock On quadratic periodic points of quadratic polynomials.
\newblock {\em ArXiv e-prints}, page arXiv:1504.00985, 2015.
\newblock Computational programs: \url{https://github.com/surim14/arith-dyn}.

\end{thebibliography}
\bibliographystyle{alpha}

\end{document}